\documentclass[11pt]{amsart}
 
\usepackage{amsmath}
\usepackage{amsfonts}
\usepackage{amssymb}
\usepackage{amsthm}
\usepackage{array}
\usepackage{chngpage}
\usepackage{comment}
\usepackage{float}
\usepackage[OT2,T1]{fontenc}
\usepackage{graphicx,caption}
\usepackage[export]{adjustbox}
\usepackage{longtable}
\usepackage{mathrsfs}
\usepackage{mathtools}
\usepackage{wrapfig}
\usepackage{cutwin}
\usepackage{shapepar}
\usepackage{tikz}
\usepackage[lowtilde]{url}
\usepackage{subcaption}
\usetikzlibrary{matrix,arrows}
\usepackage{tabularx}
\usepackage{multirow}
\usepackage[colorlinks, citecolor = blue,backref]{hyperref}
\usepackage[capitalize, nameinlink]{cleveref}
\usepackage{enumitem}
\usepackage{diagbox}

\newcommand{\commentout}[1]{}

\usepackage[margin=1.2in]{geometry}

\let\C=\undefined
\newcommand{\C}{\mathbb{C}}

\newcommand{\Q}{\mathbb{Q}}

\newcommand{\Z}{\mathbb{Z}}

\newcommand{\cC}{\mathcal{C}}

\newcommand{\calO}{\mathcal{O}}

\DeclareSymbolFont{cyrletters}{OT2}{wncyr}{m}{n}
\DeclareMathSymbol{\sha}{\mathalpha}{cyrletters}{"58}

\newcommand{\eps}{\varepsilon}

\newlength{\strutheight}
\newtheorem{theorem}{Theorem}[section]
\newtheorem{lemma}[theorem]{Lemma}
\newtheorem{corollary}[theorem]{Corollary}
\newtheorem{conjecture}[theorem]{Conjecture}
\newtheorem{proposition}[theorem]{Proposition}
\newtheorem{heuristic}[theorem]{Heuristic}
\newtheorem{question}[theorem]{Question}

\theoremstyle{definition}
\newtheorem{remark}[theorem]{Remark}
\newtheorem{figurecap}[theorem]{Figure}
\newtheorem{tablecap}[theorem]{Table}

\title{Counting modular forms by rationality field}
\author{Alex Cowan}
\address{Department of Mathematics, Harvard University, Cambridge, MA 02138 USA}
\email{cowan@math.harvard.edu}
\thanks{AC was supported by the Simons Foundation (Collaboration Grant 550031).}

\author{Kimball Martin}
\address{Department of Mathematics, University of Oklahoma, Norman, OK 73019 USA}
\email{kimball.martin@ou.edu}
\thanks{KM was supported by the Simons Foundation (Collaboration Grant 512927), the Japan Society for the Promotion of Science (Invitational Fellowship L22540), and the Osaka Central Advanced Mathematical Institute (MEXT Joint Usage/Research Center on Mathematics and Theoretical Physics JPMXP0619217849).}
\date{\today}

\begin{document}
\maketitle
\begin{abstract}
We investigate the distribution of degrees and rationality fields of weight 2 newforms.  In particular, we give heuristic upper bounds on how often degree $d$ rationality fields occur for squarefree levels, and predict finiteness if $d \ge 7$.  
When $d=2$, we make predictions about how frequently specific quadratic fields occur, prove lower bounds, and conjecture that $\Q(\sqrt 5)$ is the most common quadratic rationality field.  
\end{abstract}

\section{Introduction}

Let $S_k(N) = S_k(\Gamma_0(N))$ be the space of holomorphic cusp forms of 
weight $k$ and level $N$ with trivial nebentypus. 
For a newform $f \in S_k(N)$, denote by $K_f$ its rationality field, i.e.\ the number field generated by its Hecke eigenvalues. Define the (\textit{rationality}) \textit{degree} and \textit{discriminant} of $f$ to be the degree and discriminant of $K_f/\Q$ respectively.  In what follows, we always assume trivial nebentypus.

Weight 2 newforms are of special interest as they correspond to modular
abelian varieties, i.e., simple factors of the Jacobian $J_0(N)$ of $X_0(N)$.
Namely (the Galois orbit of) a degree $d$ newform $f \in S_2(N)$ corresponds
to a $d$-dimensional simple abelian subvariety
of $J_0(N)$ which has conductor $N^d$.  
The $d=1$ case is the celebrated bijection between
rational newforms in $S_2(N)$ and isogeny classes of elliptic curves of conductor
$N$.

It is expected (e.g., see \cite{serre, lipnowski-schaeffer, me:maeda})
that the Galois orbit of a newform is ``as large as possible'' 100\% of the time,
so that newforms have small degree rather infrequently.  On the other hand,
there are a relatively large number of elliptic curves of small conductor.
Watkins' \cite{watkins} refinement of the Brumer--McGuinness heuristics \cite{bm}
for counting elliptic curves suggests that the number of weight 2
rational newforms with level $N < X$ grows like $c X^{5/6}$ for some 
computable constant $c$.  See \cite{dk, SSW} for some theoretical results
towards this growth rate.
Note that the total number of weight 2 newforms
of level $N < X$ grows roughly like $X^2$.

Here we consider the questions: how many weight 2 newforms of level
$N < X$ are there with a given degree $d$ or
a given rationality field $K$?  There is no analogue of the 
Brumer--McGuinness heuristics for $d > 1$, since those rely on having
simple equations for elliptic curves.  Moreover, as degree $d > 1$ forms
are relatively rare, it is difficult to generate enough data to predict
precise asymptotics based on calculations.  

In fact, even for $d=1$, it is
difficult to make accurate predictions based solely on computations.  E.g., 
as remarked in \cite{watkins}, the growth rate in 
Cremona's database of elliptic curves is about $X^{0.98}$; however
more recent and very extensive calculations for prime conductors in \cite{bgr} 
align closely with the $X^{5/6}$ heuristic.

Using a combination of heuristics and data, we predict
some bounds on asymptotic orders of growth, and the relative frequency
of such forms.  

\begin{conjecture} \label{conj1}
Let $\eps > 0$.
The number of degree $d$ weight $2$ newforms
of squarefree level $N \le X$ is $O(X^{1-d/6+\eps})$ as $X \to \infty$.
In particular, this number is finite if $d \ge 7$.
\end{conjecture}

\begin{conjecture} \label{conj2}
Among squarefree levels $N \to \infty$, 100\% of degree $2$ newforms in $S_2(N)$ have rationality field $\Q(\sqrt 5)$.
\end{conjecture}

\begin{remark}
The heuristics for these conjectures do not require a restriction to
squarefree levels, however there are special considerations for
non-squarefree levels.  First, one should only count quadratic twist classes
for a more general analogue of \cref{conj1}.
Second, CM forms (which do not occur in squarefree level with trivial
nebentypus) deserve separate consideration.  
Third, if $p^r \mid N$ for sufficiently large $r$,
 then the rationality field of a newform $f$ of level $N$ must contain a
certain cyclotomic subfield (e.g., if $p \ge 5$ and $r \ge 3$, then 
$K_f \supset \Q(\zeta_p)^+$)---see \cite{brumer:rank, me:conductor}.

It is at least plausible that \cref{conj2} holds for general levels, and
\cref{conj1} holds for general levels if one restricts to counting non-CM
newforms up to quadratic twist.  However, our data are much more
limited for non-squarefree levels.
\end{remark}

\cref{conj1} is just a conjectural upper bound,
and it may not be sharp for $2 \le d \le 5$ (see below for more discussion).
When $d=1$, one can prove a lower bound of order $X^{5/6}$ for elliptic curves, but we are not aware of nontrivial
analogous lower bounds (or even a proof of infinitude!) for any $d > 1$.  
Using constructions of genus 2 curves with real multiplication, 
we obtain the following lower bounds for $d=2$, without a restriction to squarefree level.

\begin{proposition} \label{prop:intro}
The number of quadratic twist classes of weight $2$ newforms with rationality field $\Q(\sqrt 5)$ (resp.\ $\Q(\sqrt 2)$) and minimal level $N < X$ is $\gg X^{1/3}$ (resp.\ $\gg X^{2/7}$).
\end{proposition}  

The same result for squarefree levels would follow if one knew certain polynomials took on squarefree values
sufficiently often.

\begin{remark} It is not even clear for which $d \le 6$ there should exist infinitely many weight 2 newforms of squarefree level.  Constructions of genus 3 curves with real multiplication suggest it may be infinite for $d=3$---see
\cref{sec:highdeg}.  For $d = 4, 5, 6$, we have little theoretical evidence,
but our data suggest these counts are infinite at least for each $d \le 4$.
\end{remark}

We will consider two approaches to predicting counts of newforms with fixed degree or rationality field.  First, in \cref{sec:degree}, we present a heuristic using
a random model for Hecke polynomials, building off of \cite{roberts,me:maeda}.
In fact this random model naively suggests 
upper and lower bounds for counts of degree $d$ forms
on the order of $X^{1-d/6 \pm \eps}$.  However, 
it ignores any geometric considerations for the existence of degree $d$ forms,
so it is unclear how accurate this heuristic is.  Nevertheless, comparing these predictions with data at least suggests it gives an upper bound,  as asserted in \cref{conj1}.

In \cref{sec:ratfld}, we suggest an approach to predict counts of weight 2 
newforms with a given rationality field $K$ by counting moduli points for suitable abelian varieties.  In principle, this would also yield the number of counts of newforms of a fixed degree $d$, and we expect this approach should give more accurate predictions than the random Hecke polynomial model.  However, it requires more knowledge about the moduli spaces and the relation between heights and conductors than we currently possess.  We carry out some of this analysis when $d=2$, namely when $K = \Q(\sqrt D)$ for $D = 5, 8, 12, 13, 17$.
This leads to \cref{conj2}, and also suggests $cX^{3/5 - \eps}$ may be a lower bound for the total count for degree 2 newforms.
However, our analysis is not definitive enough to confidently conjecture this.

A database of all prime-level forms of degree $6$ or less and level $2\cdot 10^6$ or less was computed using an algorithm of the first author \cite{cowan}. In \cref{sec:data}, we use this database to investigate \cref{conj1} and \cref{conj2}, and pose some related questions.

\subsection*{Acknowledgements}
We are especially grateful to Noam Elkies for many insights and suggestions.  We have also benefited from conversations with Eran Assaf, Armand Brumer, Bjorn Poonen, Ari Shnidman, and John Voight.
Computations were performed at the OU Supercomputing Center for Education \& Research (OSCER) at the University of Oklahoma (OU).

\section{Counts by degree} \label{sec:degree}

First we discuss counting newforms of fixed degree. For a newform $f$,
let $\deg f = [K_f: \Q]$ be its rationality degree.  Set 
\begin{align*}
  \cC_d(X) = \#\{\text{newforms } f \in S_2(N) : \, N<X,\, N \text{ squarefree},\,  \deg f = d\}.
\end{align*}
As explained in the introduction, we restrict to squarefree $N$ for simplicity,
though our initial discussion applies equally well to counting quadratic
twist classes of non-CM weight 2 newforms.

Watkins \cite{watkins}, building on heuristics of Brumer and McGuinness \cite{bm}, formulates heuristics that suggest
\begin{align}\label{bmw_conj}
  \cC_1(X) \sim c_1 X^{\frac{5}{6}}
\end{align}
for some computable constant $c_1$.
It is known that $X^{\frac 56} \ll \cC_1(X) \ll X^{1+\eps}$ \cite{dk}.
Furthermore, Shankar--Shankar--Wang \cite{SSW} show a growth rate of
$b_1 X^{\frac 56}$ if one restricts to elliptic curves of squarefree conductor
coprime to 6 with some restrictions on discriminant-conductor and discriminant-height ratios.

While the exponent $\frac 56$ is not in clear agreement with databases of
elliptic curves in general levels (the Cremona \cite{cremona:git} and Stein--Watkins \cite{sw} databases),
the compatibility with prime-level data is much better. 
Namely, Watkins' heuristic suggests a growth rate of 
$c_1' \textrm{li}(X^{\frac 56})$ for prime levels, and this fits
extremely well with the extensive database of 
elliptic curves of prime conductor in \cite{bgr}.
Thus there is much evidence towards \eqref{bmw_conj}.

For $d > 1$, the situation is much more mysterious. In \cite{serre}, Serre proves
a statement which strongly suggests, though does not quite imply,
the bound $\cC_d(X) = o(X^2)$. 
Namely, if $N \to \infty$ along a sequence which is coprime to a fixed prime $\ell$, among bases of eigenforms for $S_2(N)$, Serre proves that the number of forms of degree $d$ is $o(\dim(S_2(N)))$ as $N \to \infty$. 
Serre's theorem was made effective by Murty and Sinha \cite{ms}, 
and more recently by Sarnak and Zubrilina \cite{sz}.

Since we do not know a good way to predict precise asymptotics for $\cC_d(X)$,
we aim to predict weaker estimates of the form
\begin{align}\label{bound_eq}
  X^{\alpha_d} \ll \cC_d(X) \ll X^{\beta_d}
\end{align}
which are nontrivial, i.e., $\alpha_d > 0$ or $\beta_d < 2$.
Computations of modular forms, as well as heuristics in
\cite{me:maeda}, suggest $\beta_d$ is decreasing in $d$, and thus we
should at least be able to take $\beta_d \le \frac 56$ for each $d \ge 1$.  In 
\cite[Question 3.1]{me:maeda}, it was also suggested that one may have
$\beta_d = 0$ for $d \gg 0$.

To our knowledge, \cref{conj1} is the first prediction of 
more precise upper bounds
(for either squarefree or general levels).  In particular, it predicts that
one can take $\beta_d = 0$ for $d \ge 7$, 
and $\beta_d$ arbitrarily small for $d=6$.  However, we do not have 
insight into whether the upper bounds in \cref{conj1} should be sharp for
$2 \le d \le 5$.

Note that \cref{conj1} implies that
$\alpha_d = 0$ is optimal among lower bounds of the form \eqref{bound_eq}
for $d \ge 6$.    In addition, \cref{prop:intro} suggests that one
may take $\alpha_2 \ge \frac 13$ for $d=2$.  (Note that \cref{prop:intro}
does not prove a lower bound for squarefree levels, only for general levels.)
This lower bound is almost certainly not sharp.  We do not have any 
predictions for lower bounds when $3 \le d \le 5$.

\subsection{Random Hecke polynomial model}
\label{sec:randhecke}

Here we present a random model to estimate the distribution of degree $d$ newforms that will lead us to \cref{conj1}.  This is based on ideas for heuristics suggested in \cite{roberts} and \cite{me:maeda}.

Consider a newspace $S_{2k}^{\mathrm{new}}(N)$.  One can further decompose this space into $2^{\omega(N)}$ joint eigenspaces of the Atkin--Lehner operators $W_p$ for $p \mid N$, which we call the Atkin--Lehner eigenspaces.  Each Atkin--Lehner eigenspace is Galois invariant.  For non-squarefree levels, one can further decompose each Atkin--Lehner eigenspace into smaller Galois invariant subspaces according to local inertia types of non-CM forms (see \cite{DPT}) and the subspace of CM forms.

For simplicity, assume $N$ is squarefree.  Then there are no CM forms of trivial nebentypus and there is only one local inertial type.  Let $S$ be an Atkin--Lehner eigenspace in $S_{2k}^{\mathrm{new}}(N)$. 
For a newform $f \in S$, the single Fourier coefficient $a_p(f)$ generates $K_f$ for 100\% of $p$ \cite{KSW}, and it is conjectured that this is true for all but finitely many $p$ if $[K_f : \Q] > 4$ \cite{murty}.
Hence, for fixed $p \nmid N$, the factorization type of the characteristic polynomial $c_{T_p}(x) \in \Z[x]$ of the Hecke operator $T_p$ will usually tell us the degrees of the newforms in $S$.  In fact, it will always give us lower bounds.

Let $n = \dim S$.  As in \cite{roberts} and \cite{me:maeda}, we can model $c_{T_p}(x)$ as a random polynomial in the set $H_n = H_n(k,p)$ of degree $n$ monic integral polynomials whose roots $\alpha$ satisfy $|\alpha| \le 2p^{k-1/2}$.  Alternatively, one can consider the set of Weil $q$-polynomials of degree $2n$ where $q = p^{k}$, or the isogeny classes of $n$-dimensional abelian varieties over $\mathbb F_{p^k}$.


Set $h(n) = \# H_n$. As discussed in \cite[\S 2.1]{me:maeda}, the number of polynomials in $H_n$ with a degree $d < \tfrac{n}{2}$ factor is approximately $h(d)h(n-d)$. Thus, if we select polynomials in $H_n$ uniformly at random, then
\begin{align}\label{eq:prob}
  \mathrm{Prob}\!\left(p \in H_n \text{ has a degree $d$ factor}\right) \approx \frac{h(d)h(n-d)}{h(n)}.
\end{align}
In this section, by approximately ($\approx$), we mean that for fixed $d$ both sides have the same growth rate in $n$ as $n \to \infty$.

For fixed $q$, no good asymptotics are known for $h(n)$ to directly estimate this probability.  There is an asymptotic for $h(n)$ when $n$ is fixed and $q$ varies.  Instead, \cite{roberts} and \cite[\S 2.1]{me:maeda} analyzed how this probability behaves if one uses a known asymptotic for $\#H_n(k,p)$ in $p^k$ when $n$ is fixed. (The result is certainly too small, as it would predict only finitely many degree 1 forms.)  

To circumvent this lack of precise asymptotics for $h(n)$ as $n \to \infty$,
we rewrite the right hand side of \eqref{eq:prob} as
\begin{align*}
  \frac{h(d)h(n-d)}{h(n)} &= h(d)\frac{h(n-d)}{h(n-d+1)} \frac{h(n-d+1)}{h(n-d+2)}  \cdots  \frac{h(n-1)}{h(n)}.
\end{align*}
One should have that $\frac{h(n-2)}{h(n-1)} \approx
\frac{h(n-1)}{h(n)}$, so applying this a small fixed number of times for a given $d$ yields
\begin{equation} \label{eq:probpower}
  \frac{h(d)h(n-d)}{h(n)} \approx \left(\frac{h(n-1)}{h(n)}\right)^d.
\end{equation}

Combining \eqref{eq:prob} and \eqref{eq:probpower} suggests that the probability of a degree $d$ factor of $c_{T_p}$ should approximately
be the $d$-th power of the probability of a degree $1$ factor of $c_{T_p}$.
The latter typically corresponds to a degree 1 form, and so we can model it using well-known expectations about counts of elliptic curves.

We will also use the following lemma.

\begin{lemma}\label{lemma:AL_dims}
Let $\nu_{2k}(X)$ be the number of Atkin--Lehner eigenspaces 
in $\bigcup_N S_{2k}^{\mathrm{new}}(N)$, where $N$ ranges over squarefree levels, having dimension less than $X$.  Then $X \ll \nu_{2k}(X) \ll X^{1+\eps}$, for any $\eps > 0$.
\end{lemma}

\begin{proof}
Since $\dim S_{2k}^{\mathrm{new}}(N) \ll N$,
the lower bound is obvious.  

Let us show the upper bound.
First, it follows from the dimension formulas for Atkin--Lehner eigenspaces from \cite{me:refdim} that any Atkin--Lehner eigenspace in $S_{2k}^{\mathrm{new}}(N)$ has dimension $\frac{(k-1)\phi(N)}{12 \cdot 2^{\omega(N)}} + O(N^{1/2 + \eps})$.  (The necessary argument, though not the statement, is given in the proof of \cite[Proposition 3.10]{me:qmf-zeroes}.)  This dimension is $\gg N^{1-\eps}$.
Hence if an Atkin--Lehner eigenspace has dimension less than $X$, it occurs in a level $N \ll X^{1+\eps}$.  Now the number of Atkin--Lehner eigenspaces in levels less than $t$ is bounded by $\sum_{N \le t} 2^{\omega(N)}$.  It is known that this latter sum is $\frac 6{\pi^2} t \log t + O(t)$.
\end{proof}

\cref{lemma:AL_dims} combines with \eqref{eq:prob} and \eqref{eq:probpower} to produce the following heuristic.

\begin{heuristic} \label{heur:main}
Suppose the number of rational newforms of weight $2k$ and squarefree level $N < X$ is $O(X^{1-\alpha})$ for some $\alpha < 1$.
Then, for any $\eps > 0$, the number of degree $d$ weight $2k$ newforms
of squarefree level $N \le X$ is $O(X^{1-\alpha d+\eps})$ as $X \to \infty$.
\end{heuristic}

Our reasoning for this heuristic is as follows.  Under the hypothetical bound $O(X^{1-\alpha})$, the lemma indicates that the probability of an Atkin--Lehner space of dimension $n$ having a size 1 Galois orbit is approximately $n^{-\alpha}$.  Assuming uniform distribution of Hecke polynomials in $H(n)$, the probability of a size $d$ Galois orbit is approximately the probability of a degree $d$ factor of $c_{T_p}$, which by  \eqref{eq:prob} and \eqref{eq:probpower} is approximately $n^{-d \alpha}$.  
Applying the lemma again leads to the stated heuristic.

Combining the Brumer--McGuinness and Watkins heuristics for $d=1$ with
\cref{heur:main} now suggests \cref{conj1} from the introduction.


\subsection{Assessment of the model}
The random model for Hecke polynomials in \cref{sec:randhecke} uses the counting measure on $H_n$, i.e., all polynomials in $H_n$ are equally likely.
If this were the case, the heuristic reasoning above would suggest both upper and lower bounds: $X^{1- \frac d6 - \eps} \ll \cC_d(X) \ll X^{1 - \frac d6 + \eps}$, so the upper bound in \cref{conj1} would be essentially optimal.

However, there are other factors controlling the distribution of Hecke polynomials  in $H_n$.  For instance, trace formulas place arithmetic conditions on the roots of Hecke polynomials.  Moreover, there are vertical and horizontal equidistribution results about convergence of the roots to Plancherel and Sato--Tate measures.
For this reason, we view our random model as a first approximation to counting degree $d$ forms.  

In \cref{empirical_degree_counts_section}, we will present data which suggests this heuristic does give an upper bound, but possibly not an optimal one.
It is really the data that lends credence to \cref{conj1}.

An alternative perspective, which we will explore in \cref{sec:ratfld}, is that
it is more natural to model the distribution of degree $d$ forms by 
modeling $d$-dimensional modular abelian varieties.  The analysis we do there 
for $d=2$ is compatible with the notion that the random Hecke polynomial heuristic gives a valid upper bound which might not be optimal.

\subsection{Finiteness questions} Related to the question of asymptotics are several questions about finiteness.  We do not investigate them here, but suggest them for future consideration.

\begin{enumerate}
\item   
We can ask: for what $d$ are there infinitely many weight 2 newforms of squarefree level?  \cref{conj1} asserts such $d$ must be at most 6, but also suggests the answer could be negative for $d=6$.  We know a positive answer for $d=1$, and expect a positive answer for $d=2$.  \cref{sec:highdeg} and our data suggest the answers may be positive for $d=3, 4$ also.

\item More generally, one can ask the same question in weight $2k$.
Roberts' conjecture \cite{roberts} implies that there are only finitely many quadratic twist classes of non-CM rational newforms in weight $2k \ge 6$.  So  \cref{heur:main} suggests that there are only finitely many newforms of squarefree level of fixed weight $2k \ge 6$ and any fixed degree $d \ge 1$.
  One might similarly expect to have finitely many quadratic twist classes of non-CM newforms of fixed degree $d$ and weight $2k \ge 6$.
  
 \item One can also ask whether there should be a uniform version 
of the finiteness part of \cref{conj1}, i.e., whether for 
sufficiently large squarefree $N$ and some fixed $d_0$ (possibly $d_0 = 6$), 
each Atkin--Lehner eigenspace 
has a unique Galois orbit of size $d \ge d_0$.  This seems plausible
based on \cref{heur:main}.  See \cref{question:lmfdb_degd} for a more precise question in prime level.
\end{enumerate}

\section{Hecke fields} \label{sec:ratfld}

In \cref{sec:degree} we considered the question of how often degree $d$ newforms occur and presented a random Hecke polynomial model, which, at least for prime levels, appears to give asymptotic upper bounds.  Here we consider the refined question of how often a specific degree $d$ rationality field $K$ should occur, and relate this question to rational points on Hilbert modular varieties.  We discuss possible lower bounds for fixed quadratic fields, prove some lower bounds, and predict that $\Q(\sqrt 5)$ is the most common quadratic rationality field.

\subsection{Modular varieties} \label{sec:modabvar}
First recall the connection between weight 2 modular forms and abelian varieties.  

Let $N \ge 1$.
To a newform $f \in S_2(N)$ with $[K_f : \Q] = d$, Shimura constructed a $d$-dimensional simple abelian variety $A_f/\Q$ satisfying the following properties.
First, $A_f$ is a quotient of $J_0(N)$.  Moreover $A_f$
is isogenous to $A_g$ if and only if $f$ and $g$ are Galois conjugates.  The endomorphism
algebra $\mathrm{End}^0(A_f) \coloneqq \mathrm{End}(A_f) \otimes \Q \simeq K_f$.  The conductor of
$A_f$ is $N^d$.  Finally, $L(s,A_f)
= \prod_\sigma L(s, f^\sigma)$, where
$f^\sigma$ ranges over the Galois conjugates
of $f$.

In general, the center of the endomorphism algebra of a $d$-dimensional abelian
variety $A$ has degree $\le d$.  If $\mathrm{End}^0(A)$ contains a totally
real field $K$ of degree $d = \dim A$, then we say $A$ has maximal real
multiplication (RM).  Any $A_f$ as above has maximal RM, and conversely
if $A/\Q$ is a simple abelian variety with maximal RM, then it is isogenous
to some $A_f$ \cite[Lemma 3.1]{me:conductor}.
%
Hence the correspondence $f \mapsto A_f$ yields a bijection between degree $d$
newforms $f$ of weight $2$ and
isogeny classes of
$d$-dimensional simple abelian varieties
$A/\Q$ with maximal RM.

We propose a heuristic approach to predicting coarse asymptotic counts of such objects.
Let $K$ be a totally real number field of degree
$d$, and $\mathfrak a$ be an ideal in $\calO_K$. The quotient $\mathfrak H^d/\mathrm{SL}(\calO_K \oplus \mathfrak a)$
parametrizes $d$-dimensional complex abelian
varieties with RM by $\calO_K$ together with a
polarization structure corresponding to $\mathfrak a$ (see \cite{goren:book} for a precise statement).  Compactifying this quotient
and desingularizing gives a Hilbert modular variety $Y(\calO_K \oplus \mathfrak a)$.  

Now consider a newform $f \in S_2(N)$ with $K_f = K$.  The abelian variety $A_f$ has endomorphism
ring an order in $\calO_K$.  Typically we expect
it is all of $\calO_K$, but if not, one can replace $A_f$ by an isogenous variety with RM by $\calO_K$.  Thus $f$ corresponds to a rational point $y$ on $Y(\calO_K \oplus \mathfrak a)$ for some $\mathfrak a$, which we can take to be in a given set of representatives for $\mathrm{Cl}^+(K)$.  

This correspondence is far from one-to-one.  First, replacing $A_f$ by
an isogenous variety, or modifying the polarization structure, may give a different 
point $y$ on $Y(\calO_K \oplus \mathfrak a)$.  Second, if $g$ is another weight 2 newform and $A_g$ is $\C$-isogenous to $A_f$, then both $f$ and $g$ correspond to the same rational points.  Third,
this is not a fine moduli space, so not all rational points on $Y(\calO_K \oplus \mathfrak a)$ will correspond to abelian varieties defined over $\Q$, and of
those that do, some will correspond to non-simple abelian varieties.

That said, it seems reasonable to expect that, generically, quadratic twist classes of Galois orbits of weight 2 newforms correspond to finite sets of rational points on $Y = \bigcup_{a \in \mathrm{Cl}^+(K)} Y(\calO_K \oplus \mathfrak a)$.  Thus one can attempt estimate the number of quadratic twist classes by estimating counts of rational points on $Y$.  A priori, it is not clear how different orderings of classes of newforms (e.g., by minimal level) will correlate with different orderings of sets of rational points (e.g., by minimal height, for some choice of height function), and we will speculate more on this for $d=2$ anon.  

Note that typically (each component of) $Y$ will be of general type, and one might expect that it has finitely many (and often no) rational points.  Hence, for a given $d$, to estimate counts of degree $d$ newforms, it should in principle suffice to consider finitely many $Y$.  Moreover, this philosophy suggests that some totally real degree $d$ rationality fields will be more common than others, roughly according to whether the moduli spaces $Y$ have many or few rational points.  Of course this is not the only consideration, due to various complications of the correspondence between newforms and rational points mentioned above.  

This philosophy is in line with Coleman's conjecture (e.g., see \cite{BFGR}), which predicts there are only finitely many isomorphism classes of endomorphism algebras for $d$-dimensional abelian varieties over $\Q$.  Hence Coleman's conjecture implies that,
for a fixed $d$, only finitely many degree $d$ rationality fields $K_f$ occur as $f$ 
varies over weight 2 newforms.


\subsection{Rational points on Hilbert modular surfaces}
\label{sec:HMS}

Now we estimate point counts on certain Hilbert modular surfaces, and pursue the ideas of the previous section for $d=2$.  

Let $D > 0$ be a fundamental discriminant and
$\calO_D$ be the ring of integers of $\Q(\sqrt D)$.
Let $Y_-(D)$ be the Hilbert modular surface constructed from the quotient $\mathfrak H^2/\mathrm{SL}(\mathcal O_D \oplus \sqrt D \mathcal O_D)$.  This parametrizes principally polarized abelian surfaces with RM by $\calO_D$ (together with a polarization structure).  See \cite{vdG}, \cite{goren:book} for details.  For brevity, we will write RM $D$ for RM by $\calO_D$.

For our heuristic point counts, we will use explicit models for Hilbert modular surfaces.  For 
$D < 100$, Elkies and Kumar \cite{EK} computed models for $Y_-(D)$.  By work of Hirzebruch and Zagier
\cite{HZ},  $Y_-(D)$ is rational (i.e., birational to $\mathbb P^2$) if and only if $D \in \{5, 8, 12, 13, 17\}$.

We expect that 100\% of degree 2 weight 2 newforms correspond to rational points on Hilbert modular surfaces with the most rational points, i.e., the rational surfaces.  While the Hilbert modular surfaces parametrizing non-principally polarizable surfaces with RM $D$ are rational over $\C$ for $D = 12, 21, 24, 28, 33, 60$ (see \cite[Theorem VII.3.3]{vdG}), we at least expect that the 5 rational $Y_-(D)$'s should account for a positive proportion of degree 2 weight 2 newforms, and this is supported by data.

To be more precise, the polarization classes of abelian surfaces with RM $D$ are in bijection with the narrow ideal classes $\mathrm{Cl}^+(\Q(\sqrt D))$.  In particular, if $D=5, 8, 13, 17$, the abelian surface is automatically principally polarizable.  For $D=12$, the narrow class number is 2, and the moduli spaces for each polarization type are rational (at least over $\C$), so it is not clear whether a positive proportion of abelian surfaces with RM $12$ should be principally polarizable.  We cannot yet analyze counts for the non-prinicipal polarization types as we do not know models for the corresponding moduli spaces together with appropriate invariants.
However, our data suggest that, at least for prime level, most degree 2 weight 2 newforms have rationality field $\Q(\sqrt D)$ with $D = 5$ or $8$, and thus correspond to points on  $Y_-(5)$ and $Y_-(8)$.

For the remainder of the section, assume $D \in \{5, 8, 12, 13, 17\}$.  Then $Y_-(D)$ is birational to $\mathbb P^2_{m,n}$.  Let $\mathcal A_2$ be the moduli space for principally polarized abelian surfaces.  Forgetting the RM action yields a map $Y_-(D) \to \mathcal A_2$.  

Let $\mathcal M_2$ be the moduli space of genus 2 curves.  To a genus 2 curve $C : y^2 = h(x)$,
one associates Igusa--Clebsch invariants $I_{2j}(C)$ for $j = 1, 2, 3, 5$.  Here $I_{2j}(C)$ can be regarded as a degree $2j$-polynomial in the coefficients of $h(x)$
, and $I_{10}(C)$ is the discriminant of $h$.
One can realize $\mathcal A_2$ as weighted projective space $\mathbb P^3_{1,2,3,5}$ with coordinates $(I_2 : I_4 : I_6 : I_{10})$.  The Torelli map $\mathcal M_2 \to \mathcal A_2$ sends the moduli of $C$ to $(I_2(C) : I_4(C) : I_6(C) : I_{10}(C))$, and the image is the complement of the hyperplane $I_{10} = 0$.
We note that Igusa--Clebsch invariants are only isomorphism invariants of $C$ up to weighted projective scaling. 

Elkies and Kumar \cite{EK} gave a birational model for $Y_-(D)$.  In particular, for generic affine coordinates $(m,n) \in \mathbb A^2$, one has an associated point on $Y_-(D)$ and thus weighted projective coordinates
$(I_2(m,n) : I_4(m,n) : I_6(m,n) : I_{10}(m,n)) \in \mathcal A_2$, where the $I_{2j}(m,n)$'s are explicit rational functions in $m, n$. 

Now we will attempt to estimate the number of rational points $(m,n)$ with bounded Igusa--Clebsch invariants. First we want to scale Igusa--Clebsch invariants (in $\mathbb P^3_{1,2,3,5}$) to be integral, as will be the case for the $I_{2j}(C)$'s given a curve $C$ over $\Z$.  Let us write $(m,n) = (a/c, b/c)$ for $a,b,c \in \Z$ with $\gcd(a,b,c) = 1$.  Regarding $a, b, c$ as variables, we scale the $I_{2j}(m,n)$'s to get polynomials $I_{2j}(a,b,c) \in \Z[a,b,c]$'s which are minimal integral over $\Z[a,b,c]$.  That is, we scale out denominators, and also any factors of the numerators $\pi$ so that $\pi^j \mid I_{2j}(a,b,c)$ for all $j \in \{1,2,3,5\}$ implies $\pi$ is a unit in $\Z[a,b,c]$.  The resulting $I_{10}(a,b,c)$'s (which are uniquely determined up to $\pm 1$) are given in \cref{tab:I10s}.\\

\noindent
\begin{tabular}{|c|c|}
\hline
$D$ & $I_{10}(a,b,c)$ \rule{0pt}{1em} \\ [0.05ex]
\hline
 \multirow{2}{*}{$5$} &
 $8(a^{5} - 10 a^{3} b^{2} + 25 a b^{4} + 5 a^{4} c - 50 a^{2} b^{2} c + 125 b^{4} c $ \rule{0pt}{1em}\\ [0.05ex]
 & $- 5 a^{3} c^{2} + 25 a b^{2} c^{2} - 45 a^{2} c^{3} + 225 b^{2} c^{3} + 108 c^{5})^{2}$ \\ [0.05ex]
\hline
 $8$ &
$8  c^{3}   (a - c)^{3}   (a + c)^{6}   (-16 a^{2} b^{2} + 32 b^{4} + a^{3} c - 56 a b^{2} c + 9 a^{2} c^{2} - 72 b^{2} c^{2} + 27 a c^{3} + 27 c^{4})^{2}$ \rule{0pt}{1em} \\ [0.05ex]
\hline
 $12$ &
$(a + c)^{3}   (a - c)^{9}   (-27 a^{2} + b^{2} + 27 c^{2})^{2}   (a^{2} b + 9 a^{2} c - 8 c^{3})^{3} $ \rule{0pt}{1em} \\ [0.05ex]
\hline
 \multirow{2}{*}{$13$} &
$2^{3} \cdot 3^{11} \cdot  (-267 a^{3} + 72 a^{2} b - a b^{2} - 3552 a^{2} c + 1440 a b c - 128 b^{2} c + 768 a c^{2})^{2} $ \rule{0pt}{1em} \\ [0.05ex]
& $\cdot \, (-12 a^{3} + 3 a^{2} c + b^{2} c)^{4}   (-a^{3} - 150 a^{2} c + 6 a b c - 264 a c^{2} + 120 b c^{2} + 64 c^{3})^{4} $ \\ [0.05ex]
\hline
 \multirow{2}{*}{$17$} &
$2^{15} \cdot 3^{11} \cdot (-132 a + b + 3 c)^{3} (-256 a^{3} - 1200 a^{2} c + 18 a b c - 6006 a c^{2} + 99 b c^{2} + 41 c^{3})^{5} $ \rule{0pt}{1em}\\ [0.05ex]
& $ \cdot \, (456 a^{2} + a b + 723 a c - 8 b c + 24 c^{2})^{3}  (4608 a^{3} - 1728 a^{2} c + b^{2} c + 216 a c^{2} - 9 c^{3})^{2} $ \\ [0.05ex]
\hline
\end{tabular}
\begin{tablecap}\label{tab:I10s}
  $I_{10}$ polynomials for $Y_-(D)$
\end{tablecap}

Specializing $a, b, c$ to integers, we denote by $I_{2j}^{\min}(a,b,c) \in \Z$ scalings which are minimal integral over $\Z$.  Note that $I_{2j}^{\min}(a,b,c) \mid I_{2j}(a,b,c)$ but they are often not equal.  E.g., when $D=5$, then the invariants $I_{2j}(1,3,2)$'s are $(-2^{4} \cdot 5^{3},
 2^{8} \cdot 5^{4},
 -2^{15} \cdot 5 \cdot 599,
 2^{21} \cdot 3^{8})$, whereas the $\Z$-minimal invariants $I_{2j}^{\min}(a,b,c)$ are obtained by scaling out a factor of $2^4$, i.e., they are $(- 5^{3}, 5^{4}, - 2^{3} \cdot 5 \cdot 599, 2 \cdot 3^{8})$.

First we want to estimate, in terms of a real parameter $T$, the growth of the cardinality of
\[Z_D(T) \coloneqq \{ (a,b,c) \in \Z^3 - U_D : \gcd(a,b,c) = 1
\text{ and } | I^{\min}_{2j}(a,b,c) | < T^{2j} \text{ for } j \in 1, 2, 3, 5 \}, \]
where $U_D$ consists of $(a,b,c)$ such that the map $(a/c, b/c) \to \mathcal A_2$ is either undefined (e.g., $c = 0$) or is not finite-to-one (e.g., for $D=8$, all points with $m=a/c=-1$ map to $(1 : 0 : 0 : 0) \in \mathcal A_2$).  
Really our interest is just in bounding $I^{\min}_{10}$, but we impose bounds on the other $I_{2j}^{\min}$'s to guarantee finiteness of $Z_D(T)$.

Precise estimates are difficult, so we make two simplifications which are sufficient to get lower bounds: 
(1) We impose the stronger bound $| I^{\min}_{2j}(a,b,c) | \le | I_{2j}(a,b,c) | < T^{2j}$. 
(2) We will suppose each monomial in $I_{2j}(a,b,c)$ is bounded by $T^{2j}$. 
Note that (1) and (2) can respectively be thought of as non-archimedean and archimedean simplifications to monomials.

\begin{proposition} \label{prop:ZDT}
We have $\# Z_D(T) \gg  T^{r_D}$ as $T \to \infty$, where
respectively $r_D = 3, \tfrac 3{2}, 2, 1, 1$ for $D = 5, 8, 12, 13, 17$.
\end{proposition}

\begin{proof}
Each $I_{2j}(a,b,c)$ is a homogeneous polynomial in $a,b,c$, say of degree $d_j$.  Taking $a, b, c$ independently up to size $T^{2j/d_j}$ shows there are $\gg T^{6j/d_j}$ tuples $(a,b,c)$ with $|I_{2j}(a,b,c)| < T^{2j}$.  Moreover, one checks the ratio $j/d_j$ is independent of the choice of $j$.  Since the conditions $\gcd(a,b,c) = 1$ and $(a,b,c) \not \in U_D$ are satisfied for a positive proportion of $(a, b, c)$, we get the asymptotic lower bound $\# Z_D(T) \gg T^{30/d}$, where $d = \deg I_{10}(a,b,c)$. 
We respectively have $d = 10, 20, 25, 30, 30$ for $D = 5, 8, 12, 13, 17$, which gives the asserted lower bounds for $D= 5, 8, 13, 17$.

For $D=12$, one can get better lower bounds using a $\mathbb P^1 \times \mathbb P^1$ parametrization for $(m,n)$.  Namely, write $(m,n) = (r/s,t/u)$ for $r,s,t,u \in \Z$.
Let $I_{2j}(r,s,t,u)$ be the minimal invariants over $\Z[r,s,t,u]$.  
These have degrees $6, 12, 18, 30$ for $j = 1, 2, 3, 5$.  Using same argument as above gives a lower bound of $\# Z_{12}(T) \gg T^{4/3}$.  
However, if we regard $I_{2j}(r,s,t,u)$ as polynomials only in $t$ and $u$, the respective degrees are $2, 4, 6, 10$ for $j = 1, 2, 3, 5$.  
Thus by taking $r, s$ uniformly bounded and $|t|, |u| \ll T$ yields $\# Z_{12}(T) \gg T^2$ as claimed.
\end{proof}

The lower bounds in the proposition are the optimal ones we could find using the $\mathbb P^2$ or $\mathbb P^1 \times \mathbb P^1$ parametrizations for $(m,n)$ by allowing either each of $a,b,c$ or $r,s,t,u$ to vary independently up to some power of $T$ (not necessarily the same power for each variable).  We remark that the optimal exponents for lower bounds using the $\mathbb P^1 \times \mathbb P^1$ parametrization for $D = 5, 8, 13, 17$ are $2, 4/3, 1, 1$, respectively.  To get these exponents, for $D=5$ one can take each of $|r|, |s|, |t|, |u| \ll T$.  For $D=8$, one takes $|r|, |s| \ll T^{2/3}$ and $|t|, |u| \ll 1$.  For both $D = 13$ and $D=17$, one takes $|r|, |s| \ll 1$ and $|t|, |u| \ll T^{1/2}$.

\begin{question} \label{qn:upper-bds}
For $D \in \{ 5, 8, 12, 13, 17\}$, is $\#Z_D(T) \ll T^{r_D+\eps}$ for any $\eps > 0$?
\end{question}

It is not clear if the simplifications to monomials affect the exponents in our estimates, but if the $I_{2j}$ polynomials are sufficiently general type, one might expect they only account for a multiplicative factor of size $O(1+T^\eps)$.  Note that in the Brumer--McGuinness heuristics, it is believed the analogous archimedean simplification (2) only affects counts by an $O(1)$ factor.

A more serious reason to doubt the exponents in these lower bounds are optimal is that there may exist (i) other rational parametrizations of $Y_-(D)$ where the $I_{2j}$ degrees are smaller, or (ii) special curves on $Y_-(D)$ which intersect $Z_D$ in an especially large number of rational points.  Indeed, the proof makes clear that different parametrizations may yield different counts for a given surface.  

Now we explain how these estimates for $\# Z_D(T)$ are related to counting quadratic twist classes of weight 2 newforms $f$ with rationality field $\Q(\sqrt D)$.  As explained above, a positive proportion of $f$ (at least if $D \ne 12$) should correspond to rational points on $Y_-(D)$.  Conversely, a rational point on $Y_-(D)$ may not come from a simple abelian surface with RM defined over $\Q$---one needs that the Mestre conic has a point, the RM is defined over $\Q$, and the Jacobian is nonsplit.  However, we expect that these obstructions will only contribute logarithmic factors to asymptotics.  (See \cite{me:rm5} for details about when the Mestre obstruction vanishes and when the RM is defined over $\Q$.)

Consider a point $(a,b,c) \in Z_D(T)$ which corresponds to a $\bar \Q$-isomorphism class of simple abelian surfaces $A/\Q$.  Generically this $\bar \Q$-isomorphism class should be the family of Jacobians of quadratic twists of a genus 2 curve $C/\Q$ with RM $D$.  Suppose this, and assume $C$ is a minimal quadratic twist of conductor $N_C$.  Then $(a,b,c)$ corresponds to the quadratic twist class of some minimal newform $f$ of level $N_f$ where $N_C = N_f^2$.  
One can write down a minimal integral model for $C$, and the polynomially-defined Igusa--Clebsch invariants $I_{2j}(C)$ are necessarily divisible by $I_{2j}^{\min}(a,b,c)$.  Thus $I_{10}^{\min}(a,b,c)$ divides the minimal discriminant $\Delta_C = 2^{-12} I_{10}(C)$ of $C$.   One also knows that the conductor $N_C \mid \Delta_C$.

Now we would like to understand how $N_C$ relates to $I_{10}^{\min}(a,b,c)$ or $I_{10}(a,b,c)$.  There are no general upper or lower bounds, and in fact there are competing issues in opposite directions.  One is that $\Delta_C$ may be much larger than either $I_{10}^{\min}(a,b,c)$ or $I_{10}(a,b,c)$, and numerically this is quite typical.  E.g., if $p^m \parallel I_{10}^{\min}(a,b,c)$ it often happens that $p^{m+10} \mid \Delta_C$ (see \cite{liu} for local results).  On the other hand, the prime powers occurring in $N_C$ are often smaller than those in $\Delta_C$.  E.g., if $p^3 \parallel \Delta_C$ then necessarily $p^2 \parallel N_C$ or $p \nmid N_C$.
Some preliminary investigations suggest that the latter issue has more impact, and asymptotic counts of curves by $I^{\min}_{10}$ may essentially be lower bounds (up to logarithmic factors) of counts by conductor.  This leads us to ask:

\begin{question} \label{heuristic:ratpts}
Let $C^{\mathrm{tw}}_{2}(D;X)$ be the number of quadratic twist classes of non-CM weight $2$ newforms of minimal level $N < X$ with rationality field $\Q(\sqrt D)$.  Is $C^{\mathrm{tw}}_{2}(D;X) \gg X^{\alpha - \eps}$ for $\alpha$ such that $\# Z_D(T) \gg T^{\alpha/5}$?  Note that one can take $\alpha = \tfrac 35, \tfrac 3{10}, \tfrac 25, \tfrac 15, \tfrac 15$ for $D = 5, 8, 12, 13, 17$. 
\end{question}

Observe that all of these exponents are less than the exponent of 2/3 from the $d=2$ case of \cref{conj1}.  We will compare these speculative lower bounds with our prime level data below. 

Even if these lower bounds hold, there are several reasons why they may not be sharp.  For one, there is the issue of \cref{qn:upper-bds}.  Perhaps most serious is the issue of how prime powers in $N_C$ relate to prime powers in $\Delta_C$ mentioned above. 

Another potential issue comes from the way we defined $Z_D(T)$: for the comparison with conductors, we are only interested in bounds on $I_{10}^{\min}$, and there may be many points with $I^{\min}_{10}$ small relative to $I_2^{\min}, I_4^{\min}, I_6^{\min}$.  In particular, for $D = 12$, if one views $I_{2j}(a,b,c)$ as a polynomial in $b$ with $a=0$ and $c$ fixed, then the degrees are $2, 2, 4, 4$, so one gets at least $T^{5/2}$ points with $I_{10}^{\min} \ll T^{10}$, which is a better than the lower bound $\gg T^{2}$ in \cref{prop:ZDT}.  These rational points correspond to the curve $m=0$ on $Y_-(12)$, and numerical investigations suggests this is a Shimura curve parametrizing abelian surfaces with geometric endomorphism algebra the quaternion algebra of discriminant 6. 
Consequently, one might be able to take $\alpha = \tfrac 12$ in \cref{heuristic:ratpts} when $D=12$.  However, it is not clear that there is a family of genus 2 curves with RM 12 over $\Q$ that would achieve $\alpha = \tfrac 12$.

\subsection{Lower bounds for quadratic fields}
There are several known families of genus 2 curves with RM 5 and RM 8.  These can be used to give lower bounds on counting such curves with bounded discriminant, and therefore conductor.  For a genus 2 curve $C/\Q$, let $\Delta_C$ denote the minimal integral discriminant.

\begin{proposition} The number of $\bar \Q$-isomorphism classes of genus $2$ curves $C/\Q$ with RM $5$ (resp.\ RM $8$) with $| \Delta_C | < X$ is $\gg X^{1/6}$ (resp.\ $\gg X^{1/7}$).
\end{proposition}

\begin{proof}
First consider RM 5.  Brumer exhibited a 3-parameter family of curves $C_{b,c,d}$ with RM 5 over $\Q$ (see \cite{brumer:rank} for an announcement and \cite{hashimoto} for a proof), however it is not clear when two such curves are isomorphic, either over $\Q$ or $\bar \Q$.  We consider the 1-parameter subfamily $C_d$ with $b=c=0$, which is given by
\[ C_d :  y^2 + (x^3 + x + 1)y = -dx^3 + x^2 + x. \]
This defines a genus 2 curve with RM 5 for all $d \in \Z$ ($I_{10}$ is never 0 for $d \in \Z$), and the discriminant of this model, $(27d^3 - 81d^2 - 34d - 103)^2$, is degree $6$ in $d$.  Thus to complete the RM 5 case of the proposition, it suffices to show that the number of $C_{d'}$ isomorphic to a given $C_{d}$ over $\bar \Q$ is finite and uniformly bounded.  

Let $I_{2j}$ (resp.\ $I_{2j}'$) denote the polynomial Igusa--Clebsch invariants for $C_d$ (resp.\ $C_{d'}$) for $j = 1, 2, 3, 5$.  Then $I_4/I_{2}^2$ and is an absolute invariant for $C_d$.  Thus if $C_d$ and $C_{d'}$ are $\bar \Q$-isomorphic, one must have $F := (I_2')^2 I_4 - I_2^2 I_4' = 0$.  Now $F=0$ defines a union of 3 curves in the $(d,d')$-plane, none of which are of the form $d=d_0$.  So the number of $C_{d'}$ which are $\bar \Q$-isomorphic to a fixed $C_d$ is bounded by $\deg F$.

Now consider RM 8.  Here we use Mestre's 2-parameter family $C'_{a,b}$ of genus 2 RM 8 curves over $\Q$ from \cite{mestre:ex}.  Consider the subfamily $C'_b = C'_{2,b}$, which is given by
\[ C'_b :  y^2 = 7500x^5 + (-75b + 3400)x^4 + (-34b + 2283)x^3 + (-3b + 1111)x^2 + 177x + 9. \]
This is a genus 2 curve with RM 8 for $b \in \Z - \{-88, 112 \}$, and the discriminant has degree 7 in $b$.  One can complete the argument just as in the RM 5 case.
\end{proof}

\begin{corollary} \label{cor:lower}
The number of quadratic twist classes of weight $2$ newforms with rationality field $\Q(\sqrt 5)$ (resp.\ $\Q(\sqrt 2)$) and minimal level $N < X$ is $\gg X^{1/3}$ (resp.\ $\gg X^{2/7}$).
\end{corollary}  

\begin{proof}
If $C$ is a genus 2 curve with RM $D$ over $\Q$ with nonsplit Jacobian, then modularity tells us that $C$ corresponds to a weight 2 newform $f$ of level $N \le \sqrt{|D_C|}$ and rationality field $\Q(\sqrt D)$, so it suffices to show that a positive proportion of the curves in the proofs have nonsplit Jacobian.  For the RM 5 $C'_d$ family above, one can check that the above model has discriminant coprime to $5$ when $d \equiv 1 \bmod 5$.  Moreover computing the $L$-polynomial shows the mod 5 Jacobian is nonsplit, whence the Jacobian of $C_d$ over $\Q$ is nonsplit.  Similarly, for the RM 8 family, the curve $C'_b$ has nonsplit Jacobian when $b \equiv 1 \bmod 7$.
\end{proof}

Note that these lower bounds are significantly smaller than those in \cref{heuristic:ratpts}.

\begin{remark}~ \label{rem:lower}
\begin{enumerate}
\item
Under the Bateman--Horn conjecture \cite{bh}, $27d^3 - 81d^2 - 34d - 103$ is prime for $\gg X/\log X$ integers $d \equiv 1 \bmod 5$.  For such $d$, $C_d$ has prime-squared discriminant and the associated newform $f$ has prime conductor.  Consequently, subject to this conjecture, the above argument shows there are $\gg X^{1/3}/\log X$ weight 2 newforms $f$ with rationality field $\Q(\sqrt 5)$ and prime level $< X$.  The analogous argument does not work for RM 8 as the discriminant of $C'_b$ splits into linear factors over $\Z$.

\item Elkies \cite{elkies:preprint} recently gave similar lower bounds for genus 2 curves with RM 5 which satisfy an Eisenstein congruence, but with the exponent $\tfrac 13$ replaced by $\tfrac 14$.

\item One should similarly be able to give explicit lower bounds for $D=12, 13, 17$ (and some other $D$).  Namely, \cite{EK} gives infinite families of RM $D$ genus 2 curves for various $D$, though without explicit models.  One can use Mestre's algorithm to construct models, and then follow the above argument.  However, this  typically yields families with discriminants of very large degree, and so one would need to do more work to get decent lower bounds.
\end{enumerate}
\end{remark}

\subsection{Remarks for higher degree}\label{sec:highdeg}
Much less is known about Hilbert modular $n$-folds for $n > 2$ than for $n=2$.  Grundman and Lippincott (\cite{GL1, GL2}) have done work towards classifying such spaces by arithmetic genus for $n = 3, 4$.  Then Borisov and Gunnells \cite{borisov-gunnells} studied the geometry of a Hilbert modular 3-fold attached to $\Q(\zeta_7^+)$.  However, to our knowledge, not much is known about the explicit geometry of higher-dimensional Hilbert modular varieties beyond these works and their references.

One case where we do know of geometric constructions leading to weight 2 modular forms of degree $d > 2$ is the following. In \cite{mestre:ex}, Mestre
constructs (among other things) families of genus $d = \frac {p-1}2$ 
hyperelliptic curves with potential RM by $\Z[\zeta_p^+]$.  When $p = 5, 7$,
the RM is actually defined over $\Q$, and generically the curves should correspond to degree $2$ and $3$ modular forms with rationality fields $\Q(\sqrt 5)$ and $\Q(\zeta_7^+)$.  

Mestre's constructions have been extended by various authors.  For instance,
 \cite{hoffman-wang} and \cite{HLSW} construct genus 3 curves $C$ with RM by an order in $\Q(\zeta_7^+)$ such that the RM is generically defined over 
 the base field.  This at least suggests there may be infinitely many weight 2 newforms of squarefree level with rationality field $\Q(\zeta_7^+)$.

\section{Data}\label{sec:data}

The LMFDB \cite{lmfdb} currently contains all weight 2 newforms of level $N \le 10^4$ \cite{bbbccdldlrsv}.
However, this range is not nearly sufficient to study asymptotic behavior of the distribution of newform degrees in squarefree level.  For instance, one is still forced to have many small degree forms of level close to 10000 as the size of Atkin--Lehner spaces can still be small.  For instance, for
$N = 9870 = 2 \cdot 3 \cdot 5 \cdot 7 \cdot 47$, there are 183 newforms in
$S_2(N)$ divided among 32 Atkin--Lehner eigenspaces, and each Atkin--Lehner eigenspace has dimension between $3$ and $8$.

One can mitigate this effect by restricting to levels with at most 2 or 3 prime factors, or ordering counts by dimensions of Atkin--Lehner eigenspaces rather than level.  However, even with such considerations, the range is still not large enough to say much about asymptotic counts.  

Instead, we analyze data we computed in prime level using the algorithms from \cite{cowan}.   Namely, we computed all weight 2 newforms of prime level less than 
$2\cdot 10^6$ and degree at most $6$, as well as the degrees of all newforms
of prime level less than $10^6$. Our data is available on the first author's personal webpage, and is currently in the process of being added to the LMFDB.

\subsection{Data for counts by degree}\label{empirical_degree_counts_section}

In \cref{deg1234_counts_best_fit_fig}, we plot counts $\cC_d'(X)$ 
of degree $d$
Galois orbits in weight 2 and prime level at most
$X$ for $1 \le d \le 4$.  The plot uses a log-log scale, and because
the number of primes up to $X$ is approximately $\text{li}(X)$, we do
a least-squares fit of functions of the form $y = \log(a\text{li}(\exp(x)^b)))$ to the data
$$\left\{ (\log X, \log \cC_d'(X)) \,:\, X < 2\cdot 10^6, \, X\text{ prime}, \, \cC_d'(X) \geq 1 \right\}.$$

\begin{figure}[t]
\centering
\includegraphics[width=0.7\textwidth]{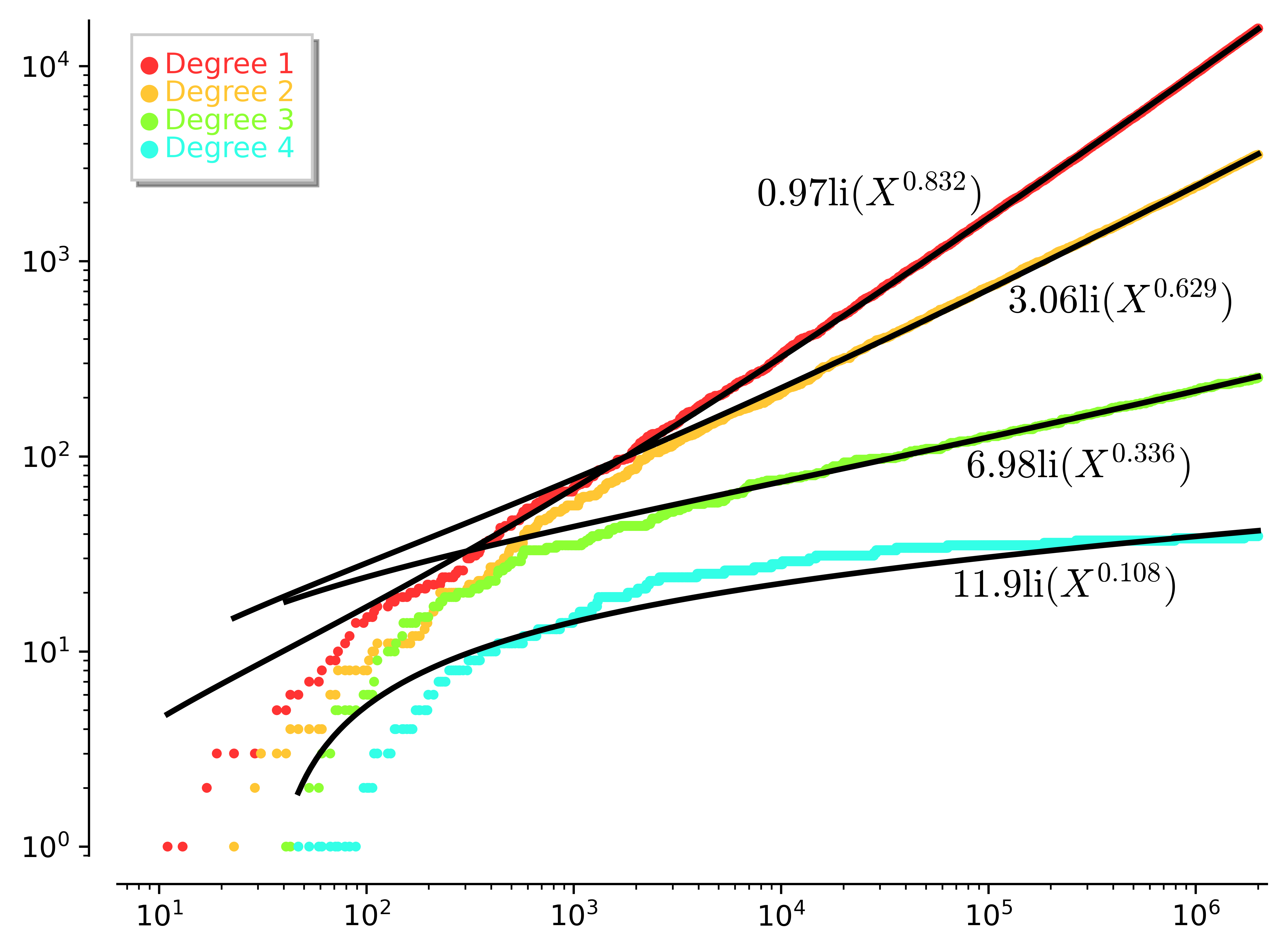}
\begin{figurecap}\label{deg1234_counts_best_fit_fig}
Number of forms with prime level less than $X$ by degree, with least-squares fits to the log-log data.
\end{figurecap}
\vspace{-\baselineskip}
\end{figure}
For degree $1$, our best fit values $a \approx 0.97$ and $b \approx 0.832$ are in agreement with the elliptic curve database \cite{bgr}, which, for $X$ up to $2\cdot 10^9$, finds that $a \approx 0.97$ and $b \approx 0.833$.  This agrees very well with the Brumer--McGuinness--Watkins heuristic (\ref{bmw_conj}).
For $d = 2, 3, 4$, the best fit exponents are approximately $0.629$, $0.336$, $0.108$.

A least-squares fit of $y = a\text{li}(x^b)$ to the data $(X,\cC_d'(X))$ yields similar values; the exponents are $0.841$, $0.622$, $0.330$, $0.107$ for $d = 1, 2, 3, 4$,
respectively.


When $1 \leq d \leq 4$, \cref{conj1} is consistent with our prime level data.  In \cref{deg1234_counts_best_fit_fig}, the growth rate $O(X^{1-d/6})$ appears to be an upper bound.  
In this range the best fit has a notably lower exponent for $d=3, 4$, but there may exist logarithmic factors in the main asymptotics for $d \ge 2$.  For example, for $d=2$, the Mestre obstruction to rationality of genus 2 curves
(see \cref{sec:HMS}) may introduce a logarithmic factor in the denominator. 
Note that in this range, $\log X > X^{\frac 16}$, so it is difficult to distinguish
between logarithmic factors and small powers of $X$.

\cref{deg_count_table} shows the number of newform orbits of degree at most $6$ with prime level between $1$ and $10^4$, between $10^4$ and $10^6$, and between $10^6$ and $2\cdot 10^6$.  Note that the first data column counts prime level forms which were already in the LMFDB.

\begin{table}[H]
\begin{tabular}{|c|r|r|r|r|}
\firsthline
\multirow{2}{*}{Degree} & \multicolumn{3}{c|}{Level range} & \multirow{2}{*}{Total} \\ [0.05ex]
\cline{2-4}
  & \multicolumn{1}{c|}{$1 \text{ -- } 10^4$} & \multicolumn{1}{c|}{$10^4 \text{ -- } 10^6$} & \multicolumn{1}{c|}{$10^6 \text{ -- } 2\cdot 10^6$} & \rule{0pt}{1em} \\ [0.05ex]
\hline
$1$ & $329$ & $8843$ & $6406$ & \rule{0pt}{1em} $15578$ \\ [0.05ex]
$2$ & $212$ & $2200$ & $1096$ & $3508$ \\ [0.05ex]
$3$ & $76$ & $142$ & $35$ & $253$ \\ [0.05ex]
$4$ & $28$ & $10$ & $1$ & $39$ \\ [0.05ex]
$5$ & $20$ & $2$ & ~ & $22$ \\ [0.05ex]
$6$ & $11$ & $1$ & ~ & $12$ \\ [0.05ex]
\hline
\end{tabular}
\begin{tablecap}\label{deg_count_table}
Number of prime-level newform orbits by degree and level. Blank entries are $0$.
\end{tablecap}
\vspace{-\baselineskip}
\end{table}

Many of the forms of degree $5$ and $6$ in this dataset have very small levels. For instance, $12$ of the $22$ degree $5$ forms are the largest-degree forms in their Atkin--Lehner eigenspaces, and similarly for $7$ of the $12$ degree $6$ forms. Only $5$ of the degree $5$ forms and $3$ of the degree $6$ forms have levels greater than $1000$.

Given this paucity of data, we refrain from any quantitative analysis of these forms, but remark that the counts for $d = 5,6$ appear to be consistent with \cref{conj1}.


\commentout{
\begin{wraptable}{r}{0.5\linewidth}
\vspace{-\baselineskip}
\centering
\begin{tabular}{|r||r|r|}
\firsthline
Level & $S_2^+(p)$ & $S_2^-(p)$ \rule{0pt}{1em} \\ [0.05ex]
\hline
$607$ & $5+7+7$ & $31$ \\ [0.05ex]
$911$ & $9+14$ & $53$ \\ [0.05ex]
$1223$ & $34$ & $9+59$ \\ [0.05ex]
$1249$ & $7+37$ & $59$ \\ [0.05ex]
$4751$ & $153$ & $18+225$ \\ [0.05ex]
\hline
\end{tabular}
\begin{tablecap} \label{table:intermediate_forms}
Atkin--Lehner eigenspaces of prime level at most $10^6$ with multiple orbits of size $\ge 7$
\end{tablecap}
\end{wraptable}
}

\cref{table:intermediate_forms} gives the decomposition type, i.e., sizes of Galois orbits, of all prime levels up to 1 million which have two or more newform orbits of degree at least $7$ in the same Atkin--Lehner eigenspace.  Here $S_2^{\pm}(p)$ denotes the subspace of $S_2(p)$ with Atkin--Lehner eigenvalue $\pm 1$ at $p$.

\begin{table}[H]
\begin{tabular}{|r|r|r|}
\firsthline
Level & $S_2^+(p)$ & $S_2^-(p)$ \rule{0pt}{1em} \\ [0.05ex]
\hline
\rule{0pt}{1em}
$607$ & $5+7+7$ & $31$ \\ [0.05ex]
$911$ & $9+14$ & $53$ \\ [0.05ex]
$1223$ & $34$ & $9+59$ \\ [0.05ex]
$1249$ & $7+37$ & $59$ \\ [0.05ex]
$4751$ & $153$ & $18+225$ \\ [0.05ex]
\hline
\end{tabular}
\begin{tablecap} \label{table:intermediate_forms}
Atkin--Lehner eigenspaces 
with multiple orbits of size at least $7$
\end{tablecap}
\vspace{-\baselineskip}
\end{table}

\cref{table:intermediate_forms} shows that there are only $6$ newforms orbits of prime level less than $10^6$ with degree $7$ or more that are not the unique largest in their Atkin--Lehner eigenspaces ($2$ of which are tied for the largest), with degrees $7$, $7$, $7$, $9$, $9$, and $18$. This data for $d \geq 7$ appears to be consistent with \cref{conj1}.

Moreover, \cref{table:intermediate_forms} shows that, for each prime level between $4751$ and $10^6$, the newforms in each Atkin--Lehner eigenspace consist of a single large Galois orbit together with orbits of size $\le 6$.
This prompts us to ask the following:
\begin{question} \label{question:lmfdb_degd}
  Is there a prime $p > 4751$ and a sign $\pm$ such that $S_2^\pm(p)$ contains two or more newform orbits each of degree $7$ or more?
\end{question}

This question can be viewed as a uniform version of \cref{conj1} in prime level. It seems plausible to us that in fact \cref{table:intermediate_forms} is a complete list of all ``mid-sized forms'' of prime level.

\subsection{Data counts by quadratic field}\label{sec:data_disc}


In this section we investigate \cref{conj2} and \cref{heuristic:ratpts} empirically. \cref{table_deg2} and \cref{disccount2_fig} present the relevant contents of our dataset.

\begin{table}[H]
\begin{tabular}{|c|r|r|r|r|}
\firsthline
\multirow{2}{*}{Disc} & \multicolumn{3}{c|}{Level range} & \multirow{2}{*}{Total} \\ [0.05ex]
\cline{2-4}
  & \multicolumn{1}{c|}{$1 \text{ -- } 10^4$} & \multicolumn{1}{c|}{$10^4 \text{ -- } 10^6$} & \multicolumn{1}{c|}{$10^6 \text{ -- } 2\cdot 10^6$} & \rule{0pt}{1em} \\ 
\hline
  $5$  & $158$ & $1900$ & $986$ & \rule{0pt}{1em} $3044$ \\ [0.05ex]
  $8$  & $37$  & $242$  & $100$ & $379$  \\ [0.05ex]
  $12$ & $1$   & $14$   & $3$   & $18$   \\ [0.05ex]
  $13$ & $13$  & $40$   & $6$   & $59$   \\ [0.05ex]
  $17$ & ~   & $1$    & ~   & $1$    \\ [0.05ex]
  \hline
\end{tabular}
\begin{tablecap}\label{table_deg2}
  Number of prime-level degree $2$ newform orbits by discriminant and level, for discriminants $D$ such that $Y_{-}(D)$ is rational. Blank entries are $0$.
\end{tablecap}
\vspace{-\baselineskip}
\end{table}

\begin{figure}[t]
  \centering
  \includegraphics[width=\textwidth]{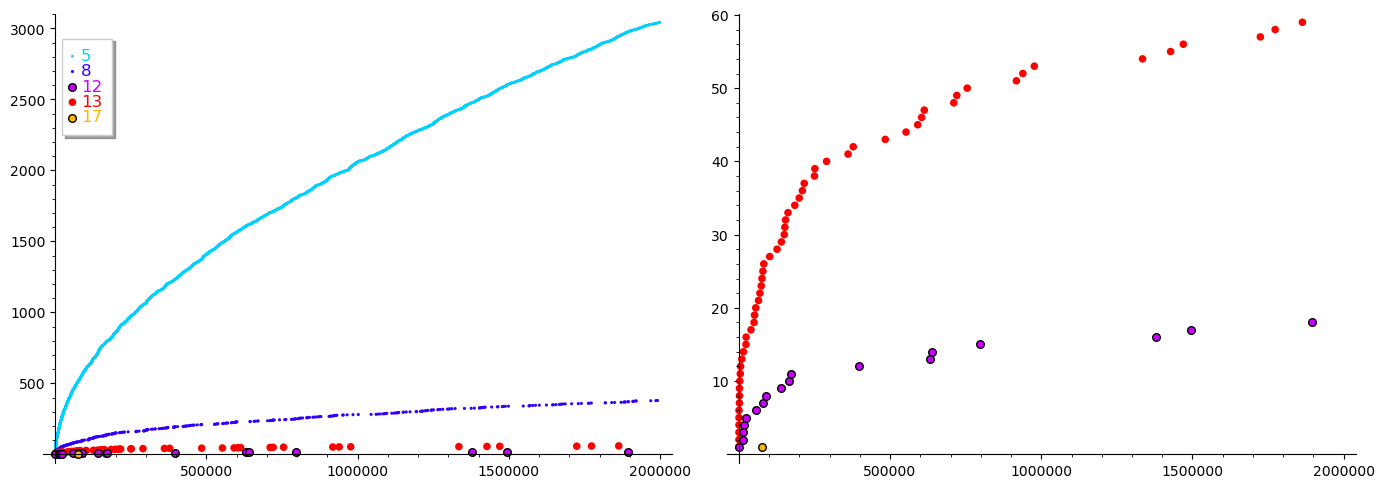}
\vspace{-0.55cm}
\begin{figurecap}\label{disccount2_fig}
Counts of degree 2 forms by discriminant $D$ for rational surfaces $Y_{-}(D)$. The plot on the right excludes discriminants $5$ and $8$.
\end{figurecap}
\vspace{-\baselineskip}
\end{figure}

The growth rate of the counts plotted in \cref{disccount2_fig} appears to be the largest for $D=5$, and $D=8$ appears to be the next largest.  This is consistent with \cref{conj2}.

Let $\cC_{d,D}'(X)$ denote the number of newform orbits of degree $d$, discriminant $D$, and prime level at most $X$. For each of $D = 5$, $8$, $12$, and $13$, we compute least-squares fits to the following four sets of data points:
\begin{enumerate}
\item $\left\{ (\log X, \log \cC_{d,D}'(X)) \,:\, 1 < X < 2\cdot 10^6, X\text{ prime}, \cC_{d,D}'(X) \geq 1 \right\}$
\item $\left\{ (\log X, \log \cC_{d,D}'(X)) \,:\, 10^4 < X < 2\cdot 10^6, X\text{ prime}, \cC_{d,D}'(X) \geq 1 \right\}$
\item $\left\{ (X, \cC_{d,D}'(X)) \,:\, 1 < X < 2\cdot 10^6, X\text{ prime}, \cC_{d,D}'(X) \geq 1 \right\}$
\item $\left\{ (X, \cC_{d,D}'(X)) \,:\, 10^4 < X < 2\cdot 10^6, X\text{ prime}, \cC_{d,D}'(X) \geq 1 \right\}$
\end{enumerate}

In the first two cases we fit functions of the form $y = \log(a\mathrm{li}(\exp(x)^b)))$, and in the last two $y = a\mathrm{li}(x^b)$. The best-fit exponents $b$ we obtain vary depending on our choice of model and range of $X$ values. We present these best-fit values of $b$ in \cref{table_deg2_fits}.

\begin{table}[H]
{\centering
\begin{tabular}{|c|r|c|c|c|c|}
\firsthline
\multirow{2}{*}{Data} & \multicolumn{1}{c|}{\multirow{2}{*}{$X$ range}} & \multicolumn{4}{c|}{Best-fit exponents} \\ [0.05ex]
\cline{3-6}
  & & \multicolumn{1}{c|}{$5$} & \multicolumn{1}{c|}{$8$} & \multicolumn{1}{c|}{$12$} & \multicolumn{1}{c|}{$13$} \\ 
\hline
  $(\log X, \log \cC_{d,D}'(X))$ & \rule{0pt}{1em} $1$ -- $2\cdot 10^6$ & $0.73$ & $0.61$ & $0.39$ & $0.42$ \\ [0.05ex]
  $(\log X, \log \cC_{d,D}'(X))$ & $10^4$ -- $2\cdot 10^6$ & $0.64$ & $0.54$ & $0.39$ & $0.38$ \\ [0.05ex]
  $(X, \cC_{d,D}'(X))$ & $1$ -- $2\cdot 10^6$ & $0.67$ & $0.56$ & $0.37$ & $0.37$ \\ [0.05ex]
  $(X, \cC_{d,D}'(X))$ & $10^4$ -- $2\cdot 10^6$ & $0.65$ & $0.53$ & $0.37$ & $0.35$ \\ [0.05ex]
  \hline
\end{tabular}
\begin{tablecap}\label{table_deg2_fits}
  Best-fit values of $b$ when fitting functions of the form $y = \log(a\mathrm{li}(\exp(x)^b)))$ or $y = a\mathrm{li}(x^b)$ as appropriate to data of counts of degree $2$ newform orbits with prime level and prescribed discriminant
\end{tablecap}
\vspace{-\baselineskip}
}
\end{table}
%

The best-fit exponents we obtain are all higher than the lower bounds proposed in \cref{heuristic:ratpts}, in many cases substantially, except for discriminant $12$, where the value is slightly lower than the $0.4$ appearing in \cref{heuristic:ratpts}. There is only one form of discriminant $17$ and prime level less than $2\cdot 10^6$, at level $75653$.

%
The data presented in \cref{table_deg2}, \cref{disccount2_fig}, and \cref{table_deg2_fits}, as well as the heuristics from \cref{heuristic:ratpts}, support \cref{conj2}.  Namely, the suggested lower bounds for counts by quadratic rationality field are largest for $\Q(\sqrt 5)$.  Since the heuristics do not rely on a restriction to prime level, one is led to ask:

\begin{question} 
Do 100\% of quadratic twist classes (ordered by minimal level) of weight $2$ degree $2$ non-CM newforms have rationality field $\Q(\sqrt 5)$?
\end{question}

There is an arithmetic reason to expect a relative scarcity of certain quadratic fields in prime level compared to the lower bounds for arbitrary levels suggested in \cref{heuristic:ratpts}.  Namely, if $C$ is a genus 2 curve with RM $D$, then we typically expect odd primes dividing $I_{10}^{\min}$ to divide the conductor $N_C$.  Based on the factorizations of $I_{10}$'s in \cref{tab:I10s}, we expect $I_{10}^{\min}$ to be a 2-power times a prime power very infrequently for $D=12,13,17$.  
Indeed, the LMFDB \cite{lmfdb} lists $5485$, $3948$ $2189$, $1230$, and $1643$ forms in all levels $N \le 10000$ for $D=5$, $8$, $12$, $13$, and $17$, respectively.  Restricting to squarefree level, these numbers are $1820$, $1124$, $445$, $319$, and $461$.

\subsection{Data counts for cubic fields}\label{sec:cubic}
While we have not attempted to carry out the approach outlined in \cref{sec:modabvar} to estimate counts of degree 3 forms with a given cubic rationality field $K$, the data, though more limited, behaves similarly as in the degree 2 case. \cref{table_deg3} and \cref{disccount3_fig} present counts of all prime level degree $3$ newform orbits by their Hecke field discriminant. These discriminants are sufficient to specify the Hecke field, in the sense that if $f$ and $g$ are degree $3$ forms of prime level less than $2\cdot 10^6$ and $\mathrm{Disc}(K_f) = \mathrm{Disc}(K_g)$, then $K_f = K_g$.

\begin{table}[H]
\begin{tabular}{|c|r|r|r|r|}
\firsthline
\multirow{2}{*}{Disc} & \multicolumn{3}{c|}{Level range} & \multirow{2}{*}{Total} \\ [0.05ex]
\cline{2-4}
  & \multicolumn{1}{c|}{$1 \text{ -- } 10^4$} & \multicolumn{1}{c|}{$10^4 \text{ -- } 10^6$} & \multicolumn{1}{c|}{$10^6 \text{ -- } 2\cdot 10^6$} & \rule{0pt}{1em} \\ 
  \hline
  $49$  & $34$ & $90$ & $30$ & \rule{0pt}{1em} $154$ \\ [0.05ex]
  $81$  & $3$  & $13$ & ~    & $16$  \\ [0.05ex]
  $148$ & $12$ & $6$  & ~    & $18$  \\ [0.05ex]
  $169$ & $2$  & $6$  & $3$  & $11$  \\ [0.05ex]
  $229$ & $8$  & $20$ & $1$  & $29$  \\ [0.05ex]
  $257$ & $9$  & $6$  & $1$  & $16$  \\ [0.05ex]
  $321$ & $2$  & $1$  & ~    & $3$   \\ [0.05ex]
  $404$ & $2$  & ~    & ~    & $2$   \\ [0.05ex]
  $469$ & $1$  & ~    & ~    & $1$   \\ [0.05ex]
  $473$ & $2$  & ~    & ~    & $2$   \\ [0.05ex]
  $621$ & $1$  & ~    & ~    & $1$   \\ [0.05ex]
  \hline
\end{tabular}
\begin{tablecap}\label{table_deg3}
  Number of prime-level degree $3$ newform orbits by discriminant and level. Blank entries are $0$.
\end{tablecap}
\vspace{-\baselineskip}
\end{table}

\begin{figure}[H]
  \centering
  \includegraphics[width=\textwidth]{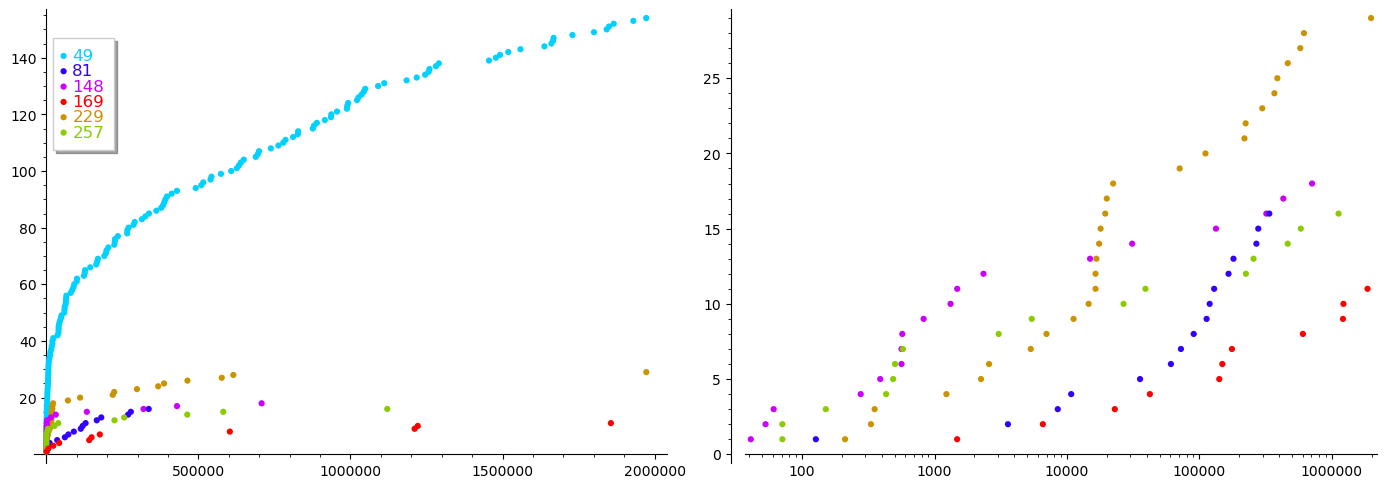}
\vspace{-0.55cm}
\begin{figurecap}\label{disccount3_fig}
Counts of degree 3 forms by discriminant. The plot on the right excludes discriminant $49$ and is on a $\log$ scale. Not shown are the $9$ forms with discriminant $321$, $404$, $469$, $473$, or $621$.
\end{figurecap}
\vspace{-\baselineskip}
\end{figure}
%

Like in \cref{sec:data_disc}, we fit functions of the form either $y = \log(a\mathrm{li}(\exp(x)^b)))$ or $y = a\mathrm{li}(x^b)$ as appropriate to data points of the form either $(\log X, \log \cC_{d,D}'(X))$ or $(X, \cC_{d,D}'(X))$, for prime $X$ between either $1$ and $2\cdot 10^6$ or $10^4$ and $2\cdot 10^6$. The best-fit exponents we obtain in each of these four cases, for $D = 49, 81, 148, 169, 229$, and $257$, are shown in \cref{table_deg3_fits}.

\begin{table}[H]
\begin{tabular}{|c|r|c|c|c|c|c|c|}
  \firsthline
  \multirow{2}{*}{Data} & \multicolumn{1}{c|}{\multirow{2}{*}{$X$ range}} & \multicolumn{6}{c|}{Best-fit exponents by $D$} \\ [0.05ex]
  \cline{3-8}
  & & \multicolumn{1}{c|}{$49$} & \multicolumn{1}{c|}{$81$} & \multicolumn{1}{c|}{$148$} & \multicolumn{1}{c|}{$169$} & \multicolumn{1}{c|}{$229$} & \multicolumn{1}{c|}{$257$} \\ 
  \hline
  $(\log X, \log \cC_{d,D}'(X))$ & \rule{0pt}{1em} $1$ -- $2\cdot 10^6$ & $0.42$ & $0.34$ & $0.12$ & $0.39$ & $0.24$ & $0.18$ \\ [0.05ex]
  $(\log X, \log \cC_{d,D}'(X))$ & $10^4$ -- $2\cdot 10^6$ & $0.42$ & $0.32$ & $0.11$ & $0.38$ & $0.20$ & $0.19$ \\ [0.05ex]
  $(X, \cC_{d,D}'(X))$ & $1$ -- $2\cdot 10^6$ & $0.43$ & $0.23$ & $0.11$ & $0.37$ & $0.19$ & $0.19$ \\ [0.05ex]
  $(X, \cC_{d,D}'(X))$ & $10^4$ -- $2\cdot 10^6$ & $0.43$ & $0.22$ & $0.11$ & $0.37$ & $0.18$ & $0.19$ \\ [0.05ex]
  \hline
\end{tabular}
\begin{tablecap}\label{table_deg3_fits}
  Best-fit values of $b$ when fitting functions of the form $y = \log(a\mathrm{li}(\exp(x)^b)))$ or $y = a\mathrm{li}(x^b)$ as appropriate to data of counts of degree $3$ newform orbits with prime level and prescribed discriminant
\end{tablecap}
\vspace{-\baselineskip}
\end{table}

Analogous to the degree $2$ case, it is natural to ask:

\begin{question} \label{qn:deg3}
Among squarefree levels $N \to \infty$, do 100\% of degree $3$ newforms in $S_2(N)$ have rationality field $\Q(\zeta_7)^+$?
\end{question}

\bibliographystyle{plain}
\bibliography{datapaperbib}{}

\end{document}